\documentclass[11pt]{amsart}
\usepackage{amsmath,amsfonts,amsthm,mathrsfs,amssymb,amscd,comment,enumerate,amsxtra,url,tikz-cd}
\input{xypic}
\xyoption{all}

\usepackage{xcolor} 
\colorlet{mdtRed}{red!50!black}
\definecolor{dblue}{rgb}{0,0,.6}
\usepackage[colorlinks]{hyperref}
\hypersetup{linkcolor=blue,citecolor=dblue,filecolor=dullmagenta,urlcolor=mdtRed}

\numberwithin{equation}{section}
\newtheorem{theorem}[equation]{Theorem}
\newtheorem{corollary}[equation]{Corollary}
\newtheorem{lemma}[equation]{Lemma}
\newtheorem{proposition}[equation]{Proposition}
\newtheorem{definition}[equation]{Definition}

\newtheorem*{theorem*}{Theorem}
\newtheorem*{corollary*}{Corollary}
\newtheorem*{proposition*}{Proposition}

\theoremstyle{remark}
\newtheorem{remark}[equation]{Remark}

\def\subsection{
	\refstepcounter{equation}
	\noindent {\bf \arabic{section}.\arabic{equation}.}
}

\newcommand{\R}{\mathbb{R}}

\newcommand{\mf}[1]{\mathfrak{#1}}

\newcommand{\mb}[1]{\mathbb{#1}}
\newcommand{\mc}[1]{\mathcal{#1}}

\newcommand{\ve}{\varepsilon}

\usepackage{marginnote}
\usetikzlibrary{calc}

\begin{document}

\title[Seshadri constants on some Quot schemes]{Seshadri constants on some Quot schemes}

\author[C. Gangopadhyay]{Chandranandan Gangopadhyay}

\address{Department of Mathematics, Indian Institute of Technology Bombay, Powai,
Mumbai 400076, Maharashtra, India.}

\email{chandra@math.iitb.ac.in}

\author[K. Hanumanthu]{Krishna Hanumanthu}

\address{Chennai Mathematical Institute, H1 SIPCOT IT Park, Siruseri, Kelambakkam 603103, India}

\email{krishna@cmi.ac.in}

\author[R. Sebastian]{Ronnie Sebastian}

\address{Department of Mathematics, Indian Institute of Technology Bombay, Powai,
Mumbai 400076, Maharashtra, India.}

\email{ronnie@math.iitb.ac.in}



\thanks{The second author was partially supported by a grant from Infosys
Foundation and DST SERB MATRICS grant MTR/2017/000243.}

\date{\today}

\begin{abstract}
Let $E$ be a vector bundle of rank $n$ on $\mathbb{P}^1$. 
Fix a positive integer $d$.  Let $\mc Q(E,d)$ denote the Quot scheme
of torsion quotients of $E$ of degree $d$ and let $Gr(E,d)$ denote 
the Grassmann bundle that parametrizes the $d$-dimensional quotients of the 
fibers of $E$. 
We compute Seshadri constants of ample line bundles 
on $\mc Q(E,d)$ and $Gr(E,d)$. 
\end{abstract}
\maketitle

\section{Introduction}

Let $X$ be a projective variety over an algebraically closed field $k$ 
and let $L$ be a nef line bundle on $X$. For a point $x \in X$, 
the \textit{Seshadri constant of  $L$ at $x$} is defined as
$$\varepsilon(X,L,x)\,:=\, \inf\limits_{\substack{x \in C}} \frac{L\cdot
C}{{\rm mult}_{x}C}\, ,$$
where the infimum is taken over all curves in $X$ passing through $x$. Here 
$L\cdot C$ denotes the intersection multiplicity and ${\rm mult}_x C$ denotes the 
multiplicity of $C$ at $x$. When there is no confusion about 
$X$, to simplify notation we denote $\ve(X,L,x)$ by $\ve(L,x)$.

Seshadri constants were introduced by Demailly \cite{D90} 
as a way to tackle the Fujita Conjecture. He was motivated by 
an ampleness criterion of Seshadri \cite[Theorem I.7.1]{H70}. 
These constants turned out to be important invariants associated
to projective varieties. 
As an illustration, let $X$ be a Fano variety of 
dimension $N$ and let $K_X$ denote the canonical bundle on $X$. 
Then $X$ is isomorphic to $\mathbb{P}^N$ if and 
only if there exists a smooth point $x \in X$ such that 
$\varepsilon(X,-K_X,x) > N$; see 
\cite{BS09,LZ18}. A lot of research is currently 
focused on studying Seshadri constants. 


Seshadri constants on surfaces have been the primary 
focus of researchers, but some results are known in general {only}
for special classes of varieties. For example, several 
results are known in the case of Fano varieties \cite{L03}, 
abelian varieties \cite{Laz96,N96,B98} or toric varieties \cite{I14}.

In this paper, we compute Seshadri constants for ample line 
bundles on Quot schemes over $C: = \mathbb{P}^1$.  
Let $E$ be a vector bundle of rank $n$ 
on $C$ and let $d$ be a positive integer. 
Let 
$\mc Q(E,d)$ denote the Quot scheme of 
torsion quotients of $E$ of degree $d$. 
The Nef cone 
of $\mc Q(E,d)$ is described in \cite{GS20}. 
In \cite{Str} the author restricts his attention 
to the case when $E$ is the trivial bundle, but considers 
more general Quot schemes (not just torsion quotients). 
We use this description to compute the Seshadri 
constants of ample line 
bundles on $\mc Q(E,d)$. 

When $E$ is the trivial bundle of rank $n$ we 
will also denote $\mc Q(E,d)$ by $\mc Q$ or $\mc Q(n,d)$.
In Section \ref{sect1}, we consider the Quot 
scheme $\mc Q(n,d)$ where $d>1$. 
The Nef cone is generated by {naturally defined} line bundles 
$[\mc L_{d-1}]$ and $[\Phi^*\mc O_{\mb P^d}(1)]$ 
(see discussion preceding \eqref{description-nefcone-trivial} 
for notation and more details).
The main result of Section \ref{sect1} is the following (see 
Definition \ref{def-Z} for the definition of 
the closed subset $\mc Z$ and the open set $\mc U=\mc Q\setminus \mc Z$).
\begin{theorem}
	Fix integers $a,b>0$.
	If $x\in \mc U\subset \mc Q$ we have 
	$$\varepsilon(a\cdot [\mc L_{d-1}]
	+b\cdot [\Phi^*\mc O_{\mb P^d}(1)],x)=a.$$ 
	If $x\in \mc Z=\mc Q\setminus \mc U$ we have
	$$\varepsilon(a\cdot [\mc L_{d-1}]
	+b\cdot [\Phi^*\mc O_{\mb P^d}(1)],x)= {\rm min}\{a,b\}.$$
\end{theorem}
In Section 
\ref{sect2}, we deal with the case $\mc Q':=\mc Q(E,d)$ where
$E$ is an arbitrary vector 
bundle and $d > 1$. 
The Nef cone is generated by line bundles 
$[\mc L_{L,\mc Q'}]$ and $[\Phi'^*\mc O_{\mb P^d}(1)]$, 
(see discussion preceding \eqref{description-nef-general-bundle} for notation and 
more details).
There is a natural inclusion $j:\mc Q'\hookrightarrow \mc Q$.
The main result of Section \ref{sect2} is the following. 

\begin{theorem}
	Assume $d>1$.
	Fix integers $a,b>0$. If $x\in \mc Q'$ such that $j(x)\in \mc U$, then we have 
	$$\varepsilon(a\cdot [\mc L_{L,\mc Q'}]
	+b\cdot [\Phi'^*\mc O_{\mb P^d}(1)],x)=a.$$ 
	If $x\in \mc Q'\setminus \mc U$ we have
	$$\varepsilon(a\cdot [\mc L_{L,\mc Q'}]
	+b\cdot [\Phi'^*\mc O_{\mb P^d}(1)],x)= {\rm min}\{a,b\}.$$
\end{theorem}

The case $d=1$ corresponds to projective bundles.
This is dealt with in Theorem \ref{main-projective-bundle}. 
Finally, using similar methods we compute Seshadri constants on 
Grassmann bundles over $C$ in Section \ref{sect4}, 
see  Theorem \ref{main-grassmann}.

If $L$ is an ample line bundle on an $N$-dimensional projective variety $X$, we have $0 < \varepsilon(X,L,x) \le \sqrt[N]{L^N}$ for all $x \in X$. 
So we define the following: 
$$\varepsilon(X,L,1) : = \sup\limits_{x\in X} \varepsilon(L,x), \text{and}$$ 
$$\varepsilon(X,L) \,:=\, \inf\limits_{\substack{x \in X}} \varepsilon(L,x)\, .$$

While $\varepsilon(X,L)$ can be arbitrarily small by an 
example of Miranda \cite[Example 5.2.1]{Laz}, it is conjectured that we 
always have $\varepsilon(X,L,1) \ge 1$ when $k$ is the field of complex numbers; see \cite[Conjecture 5.2.4]{Laz}.  
This is known to be true when $X$ is a surface (see \cite{EL93}), but in higher dimension it is open, in general. 
As a consequence of our main results, we conclude that. 
\begin{corollary}
The inequality $\varepsilon(X,L,1) \ge 1$ 
holds for ample line bundles $L$ when $X$ is $\mc Q$ or $\mc Q'$.	
\end{corollary}

\section{Seshadri constants on $\mc Q(n,d)$}\label{sect1}
Throughout we assume the base field to be an algebraically closed field $k$. 
In this section, we consider the Quot scheme associated to the trivial bundle. The results 
proved in this section will be used when we consider the general case in Section \ref{sect2}. 

Let $C:=\mb P^1$. Fix integers $n, d \geq 1$. Let 
$\mc Q:=\mc Q(n,d)$ be the Quot scheme of torsion quotients 
of the vector bundle $\mc O^n_C$ of degree $d$.
{It is well known that $\mc Q$ is a smooth projective variety of dimension $n\cdot d$. 
Let $ \mc O^n_{C\times \mc Q}\to \mc B\to 0$ be the 
universal quotient over $C\times \mc Q$. This quotient has the following universal property: 
Let $T$ be a finite type scheme over $k$. 
Suppose we have a quotient 
$\mc O^n_{C\times T}\to B_T\to 0$ over $C\times T$ 
such that 
$B_T$ is flat over $T$ and 
for all closed points $t\in T$, 
the sheaf $B_T|_{C\times t}$ is a torsion sheaf of degree $d$. 
Then there exists a unique map $f_T:T\to \mc Q$ such that 
$$(id_C\times f_T)^*[\mc O^n_{C\times \mc Q}\to \mc B]=\mc O^n_{C\times T}\to B_T\,.$$}

In this section we compute Seshadri constants of ample line bundles 
on $\mc Q$. 
We begin by describing the Nef cone of $\mc Q$. 

Let $p_C:C\times \mc Q\to C$ and 
$p_{\mc Q}:C\times \mc Q\to \mc Q$ be the projections.   
The sheaf ${p_{\mc Q}}_*(\mc B)$ is a vector bundle over
$\mc Q$ of rank $d$. Define the line bundle 
\begin{equation}
	\mc O_{\mc Q}(1):={\rm det}({p_{\mc Q}}_*(\mc B))\,.
\end{equation}
Let 
\begin{equation}
	\Phi:\mc Q\to S^d\mb P^1\cong \mb P^d
\end{equation}
be 
the Hilbert-Chow map (see \cite[Section 2]{GS19}).
This map has the following pointwise description. 
If $[\mc O_C^n\to B]$ is a torsion quotient of 
degree $d$, then $\Phi$ maps this to the point 
in $S^d\mb P^1$ corresponding to 
$\sum_{x\in {\rm Supp}(B)}l(B_x)[x],$
where $l(B_x)$ denotes the length of the $\mc O_{C,x}$-module $B_x$.

The N\'eron-Severi space $N^1(\mc Q)$ is two-dimensional
and the classes $[\mc O_{\mc Q}(1)]$ and $[\Phi^*\mc O_{\mb P^d}(1)]$
form a basis (see \cite[Corollary 3.10]{GS20}). Define 
\begin{align}
	\mc L_{d-1}&:= \mc O_{\mc Q}(1)\otimes \Phi^*\mc O_{\mb P^d}(d-1)\,.
\end{align}
The nef cone 
${\rm Nef}(\mc Q)\subset N^1(\mc Q)$ of $\mc Q$ is the 
cone generated by the following classes (see \cite[Proposition 6.1]{GS20})
\begin{equation}\label{description-nefcone-trivial}
	{\rm Nef}(\mc Q)=\R_{\geq 0} [\mc L_{d-1}]+
		\R_{\geq 0}[\Phi^*\mc O_{\mb P^d}(1)]\,.
\end{equation}
For a line bundle $L$ on $C$ we denote
\begin{align}
	\mc L_{L,\mc Q}&:={\rm det}(p_{\mc Q*}(\mc B\otimes  p_C^*L))\,.
\end{align}

\begin{lemma}\label{lemma-L_{1,x}}
	Given any point $x\in \mc Q$ there exists a curve 
	$L_{1,x}\cong \mb P^1  \hookrightarrow \mc Q$ passing through $x$
	such that 
	$[\mc L_{d-1}]\cdot [L_{1,x}]
	= 1$ 
	and 
	$[\Phi^*\mc O_{\mb P^d}(1)]\cdot [L_{1,x}]=0$.
\end{lemma}

\begin{proof}
Let $E=\mc O_C^n$.
Let $x\in \mc Q$ correspond to the quotient $x:E\to B\to 0$.
Fix a quotient $B\to B'\to 0$ where $B'$ is a torsion sheaf of degree
$d-1$. Let $A$ be the kernel of $E\to B\to 0$ and let $A'$ be the kernel of 
the composition $E\to B\to B'\to 0$. Then $A\subset A'\subset E$ and
we have an exact sequence
$$0 \to A'/A \to B \to B'\to 0\,.$$
Hence, $A'/A$ is a torsion sheaf of degree $1$, that is,
$A'/A\cong k_{c}$, where $k_c$ is the skyscraper sheaf with fibre $k$ at a point
$c\in C$. Let $A'_c$ be the fiber of the 
sheaf $A'$ over the point $c\in C$ and 
let $Z:=\mb P(A'_c)$. 
Let $p_1:C\times Z\to C$ and $p_2:C\times Z\to Z$
be the projections. Let $i:c\times Z\hookrightarrow C\times Z$ be the inclusion.
Then we define a quotient on $C\times Z$ as the composition
\begin{equation}\label{eqn-fibre of S_d}    
{p^*_1A'\to i_*i^*p^*_1A'\cong i_*(A'_c\otimes \mc O_Z)\to i_*\mc O_Z(1)\,.}
\end{equation}
Let $A_Z\subset p^*_1A'\subset p_1^*E$ be the kernel of the above composition. 
Let us denote the quotient $p_1^*E/A_Z$ by $B_Z$. 
Then we have an exact sequence on $C\times Z$
\begin{equation}\label{eqn-quotient over S_d}    
0 \to i_*\mc O_Z(1) \to B_Z\to p^*_1B'\to 0\,.
\end{equation}    
Hence $B_Z$ is flat over $Z$ such that 
$B_Z|_{C\times z}$ is a torsion sheaf of degree $d$ for every $z \in Z$. 
Therefore the quotient $p_1^*E\to B_Z\to 0$ defines a map 
$f:Z\to \mc Q$. It is clear that $x\in \mc Q$ is in the image of this 
map. Using \cite[Lemma 3.1 (ii)]{GS20} and 
(\ref{eqn-quotient over S_d}) we get that 
$f^*\mc O_{\mc Q}(1)=\mc O_Z(1)$. Let $z\in Z$. Then the quotient 
corresponding to $f(z)$ sits in the short exact sequence,
$$0\to k_c\to B_Z\vert_z\to B'\to 0\,,$$
obtained by restricting \eqref{eqn-quotient over S_d}
to $C\times z$. This shows that the divisor 
${\rm div}(B_Z\vert_z)$
corresponding to 
$B_Z\vert_z$  is the sum ${\rm div}(B')+c$. 
Thus, $\Phi\circ f$ is constant;  that is, the image of $f$ 
is contained in a fiber of $\Phi$. 

Next we will show that $f$ is a closed immersion. 
Note that $E$ is globally generated. Let 
$V:=H^0(C,E)$. Then we have a surjection 
$V\otimes \mc O_C \to E$. This gives surjective maps
\begin{equation}\label{e-1}
V\otimes p_C^*\mc O_C\twoheadrightarrow p^*_CE\twoheadrightarrow \mc B\,.
\end{equation}
Letting $K$ denote the kernel of the map $V\otimes p_C^*\mc O_C\to \mc B$
and observing that $K\vert_q$ decomposes into a direct sum
of line bundles $\mc O_C(b_i)$ with $-d\leq b_i\leq 0$,
one easily checks that $R^1p_{\mc Q*}(K\otimes p_C^*\mc O(d-1))=0$.
Let $L:=\mc O_C(d-1)$.
Tensoring \eqref{e-1} with $p_C^*L$, applying $p_{\mc Q*}$
and taking exterior power, we see that 
\begin{equation}\label{e-2}
{\rm det}(p_{\mc Q*}(\mc B\otimes  p_C^*L))=\mc L_{L,\mc Q}
\text{ is a globally generated line bundle on $\mc Q$.}
\end{equation}
In \cite[Proposition 6.1]{GS20} it is proved that 
$$\mc L_{L,\mc Q}\cong \mc O_{\mc Q}(1)\otimes  \Phi^*\mc O_{\mb P^d}(d-1)=\mc L_{d-1}\,.$$ 
Since $\Phi\circ f$ is constant, it follows that 
$f^*\mc L_{L,\mc Q}\cong f^*\mc O_{\mc Q}(1)\cong \mc O_Z(1)$.
Consider the composite map $Z\to \mc Q\to \mb P^N$, where 
the second map is given by $\mc L_{L,\mc Q}$. Since the pullback
of $\mc O_{\mb P^N}(1)$ along this map is $\mc O_Z(1)$ it follows
that the composite, and so also $f$, is a closed immersion.

Let $L_{1,x}\subset Z\subset \mc Q$ denote any line passing through $x$.
Then it is clear that $L_{1,x}$ satisfies the assertions of the lemma. 
\end{proof}

Let $\mc O^n_{C\times \mc Q}\to \mc B$ denote the universal quotient 
on $C\times \mc Q$. Pushing this forward we get a map 
$\mc O_{\mc Q}^n\to p_{\mc Q*}\mc B$ of sheaves on $\mc Q$. Let $\mc F$
denote the cokernel. From Grauert's Theorem it is clear that 
for $q\in \mc Q$ the fiber $\mc F\otimes k(q)$ is the cokernel 
of the map $H^0(C,\mc O^n_C)\to H^0(C,\mc B_q)$. Note that the image 
of this map cannot be 0, thus, ${\rm dim}(\mc F\otimes k(q))\leq d-1$
for all points $q\in \mc Q$.
The set 
$$\{q\in \mc Q\,\vert\, {\rm dim}(\mc F\otimes k(q))\geq d-1 \}=
	\{q\in \mc Q\,\vert\, {\rm dim}(\mc F\otimes k(q))= d-1 \}$$
is a closed subset.

\begin{definition}\label{def-Z}
	Define $\mc Z \subset \mc Q$ 
	to be the closed set consisting of points $q$ for which 
	the image of $H^0(C,\mc O^n_C)\to H^0(C,\mc B_q)$
	is $1$-dimensional.
	Define $\mc U:=\mc Q\setminus \mc Z$.
\end{definition}

\begin{lemma}\label{lemma-L_{2,x}}
	Given any point $x\in \mc Z$ there exists a curve 
	$L_{2,x} \cong \mb P^1 \hookrightarrow \mc Q$ passing through $x$ 
	such that 
	$[\mc L_{d-1}]\cdot [L_{2,x}]=0$ and
	$[\Phi^*\mc O_{\mb P^d}(1)]\cdot [L_{2,x}]=1$.
\end{lemma}

\begin{proof}
	Let $x\in \mc Z$. Let $0 \neq w\in H^0(C,B)$ be an element in the image of the map  
	$H^0(C,\mc O^n_C)\to H^0(C,B)$. Then the quotient corresponding to $x$ factors as
	$\mc O^n_C\xrightarrow{v} \mc O_C \xrightarrow{w} B$. 
	Associated to the surjection $v:\mc O^n_C\to \mc O_C$ 
	we have a section $\eta_v:\mb P^d\hookrightarrow \mc Q$ of $\Phi$ 
	as in \cite[Equation (3.12)]{GS20}
	which passes through the point $x$ such that 
	$$\eta_v^*\mc L_{d-1}=\eta_v^*(\mc O_{\mc Q}(1)\otimes \Phi^*\mc O_{\mb P^d}(d-1))\cong \mc O_{\mb P^d}\,.$$
	This is explained in the second paragraph of the proof in \cite[Proposition 6.1]{GS20}.
	Now choose any line in $\mb P^d$ such that its image under $\eta_v$, call it
	$L_{2,x}$, passes through $x$.  
	Then it is clear that $L_{2,x}$ has the required properties.        
\end{proof}

Recall, as remarked after equation \eqref{e-1}, that the natural map 
$$H^0(C,\mc O_{C}(d-1)) \otimes \mc O^n_{\mc Q}
=p_{\mc Q*}(\mc O^n_{C\times \mc Q}\otimes p_C^*\mc O_{C}(d-1))
\to p_{\mc Q*}(\mc B\otimes  p_C^*\mc O_{C}(d-1))$$ 
is surjective. 
Let ${\rm Gr}(H^0(C,\mc O_{C}(d-1))^n,d)$
denote the Grassmannian of $d$-dimensional quotients 
of the vector space $H^0(C,\mc O_{C}(d-1))^n$.
Hence we have morphisms 
$$\mc Q\to {\rm Gr}(H^0(C,\mc O_{C}(d-1))^n,d)
\hookrightarrow \mb P^N,$$
where the second map is the Pl\"ucker embedding.
We will denote the first map by $\psi$ and the composition by $\Psi$. Note that, as remarked earlier,
we have an isomorphism
$\Psi^*\mc O_{\mb P^N}(1)\cong \mc L_{d-1}$. 

\begin{lemma}\label{lemma-an injective map from an open set}
	The map 
	$\Psi|_{\mc U}:\mc U\to \mb P^N$ is injective.
\end{lemma}

\begin{proof}
	It is enough to show that the map 
	$$\psi|_{\mc U}:\mc U \to {\rm Gr}(H^0(C,\mc O_{C}(d-1))^n,d)$$
	is injective. By definition, if $x\in \mc U$ corresponds to the quotient 
	$x:\mc O^n_C\to B\to 0$ then the image of $x$ under the map $\psi$ is the quotient
	of vector spaces 
	$$H^0(C,\mc O_{C}(d-1))^n\to H^0(B\otimes \mc O_{C}(d-1))\to 0\,.$$
	Let $A=\bigoplus\limits^n_{i=1} \mc O_C(d_i)$ be the kernel of $x$. 
	We will show that $A\otimes \mc O_{C}(d-1)$ is globally generated, 
	which is equivalent to showing $d_i\geq -(d-1)$.  Let us assume the 
	contrary, that is, 
	there exists $d_i$ such that $d_i\leq -d$. Then 
	$$h^1(C,A)\geq h^1(C, \mc O_{C}(d_i))=h^0(C,\mc O_{C}(-2-d_i))
	\geq h^0(C,\mc O_{C}(-2+d))=d-1\,.$$
	Now consider the exact sequence 
	$$H^0(C,\mc O^n_C)\to H^0(C,B) \to H^1(C,A) \to 0\,.$$
	Since $x\in \mc U$ we have that the image of the first map has dimension $\geq 2$.
	Therefore, $h^1(C,A)\leq d-2$ and we arrive at a contradiction. Hence 
	$A\otimes \mc O_{C}(d-1)$ is globally generated. This means that
	$A$ as a subsheaf of $\mc O^n_C$ (and so also the quotient $x$) 
	can be recovered from the map 
	$H^0(C,A\otimes \mc O_{C}(d-1))\to H^0(C,\mc O_{C}(d-1))^n$
	by taking the sheaf generated by the sections 
	$H^0(C,A\otimes \mc O_{C}(d-1))$ in $\mc O_{C}(d-1)^n$ and then twisting by 
	$\mc O_C(-d+1)$. Thus, we get that the required map is injective. 
\end{proof}

\begin{theorem}\label{trivial-main}
	Fix integers $a,b>0$.
	If $x\in \mc U\subset \mc Q$ we have 
	$$\varepsilon(a\cdot [\mc L_{d-1}]
	+b\cdot [\Phi^*\mc O_{\mb P^d}(1)],x)=a.$$ 
	If $x\in \mc Z=\mc Q\setminus \mc U$ we have
	$$\varepsilon(a\cdot [\mc L_{d-1}]
	+b\cdot [\Phi^*\mc O_{\mb P^d}(1)],x)= {\rm min}\{a,b\}.$$
\end{theorem}

\begin{proof}
	Let $x\in \mc U$. 
	Let $X$ be any irreducible and reduced curve in $\mc Q$ passing through $x$.
	Then by Lemma \ref{lemma-an injective map from an open set} the map
	$\Psi(X\cap \mc U)\neq {\rm pt}$. Hence there exists a section
	$H\in H^0(\mc Q,\mc L_{d-1})$ passing through $x$ such that $X$ is not
	contained in $H$. By B\'ezout's theorem, we get     
	$$[\mc L_{d-1}]\cdot [X]
	=[H]\cdot [X]\geq {\rm mult}_xX\,.$$
	Hence    
	$$(a\cdot [\mc L_{d-1}]+b\cdot [\Phi^*\mc O_{\mb P^d}(1)])\cdot [X]
	\geq a\cdot [\mc L_{d-1}] \cdot [X]\geq a\cdot {\rm mult}_xX\,.$$
	Therefore, $\varepsilon(a\cdot [\mc L_{d-1}]
	+b\cdot [\Phi^*\mc O_{\mb P^d}(1)],x)\geq a$.
	Now  by Lemma \ref{lemma-L_{1,x}} we have 
	$$\varepsilon(a\cdot [\mc L_{d-1}]
	+b\cdot [\Phi^*\mc O_{\mb P^d}(1)],x)
	\leq \dfrac{(a\cdot [\mc L_{d-1}]
		+b\cdot [\Phi^*\mc O_{\mb P^d}(1)])\cdot [L_{1,x}]}{{\rm mult}_xL_{1,x}}
	=a\,.$$
	Hence, we get for $x\in \mc U$
	$$\varepsilon(a\cdot [\mc L_{d-1}]
	+b\cdot [\Phi^*\mc O_{\mb P^d}(1)],x)=a\,.$$    
	
	Now let $x\in \mc Z$.
	We first prove the inequality 
	$$\varepsilon(a\cdot [\mc L_{d-1}]
	+b\cdot [\Phi^*\mc O_{\mb P^d}(1)],x)\geq {\rm min}\{a,b\}\,.$$  
	Let $X$ be any irreducible and reduced curve in $\mc Q$ 
	passing through $x$.  
	We have maps $\Psi:\mc Q \to \mb PH^0(\mc Q,\mc L_{d-1})$ 
	and $\Phi:\mc Q\to \mb P^d$.
	The class $[\mc L_{d-1}]+[\Phi^*\mc O_{\mb P^d}(1)]$ is ample,
	thus, it cannot happen that  
	$\Psi$ and $\Phi$ both are constant on $X$. 
	First consider the case when $\Psi$ is non-constant on $X$.
	Hence, there 
	exists a section 
	$H_1\in H^0(\mc Q,\mc L_{d-1} )$ passing 
	through $x$ such that $X$ is not contained in $H_1$. 
	By B\'ezout's theorem, we will have
	$$[\mc L_{d-1}]\cdot [X]
	=[H_1]\cdot [X]\geq {\rm mult}_xX\,,$$
	which gives 
	$$\frac{(a\cdot [\mc L_{d-1}]
		+b\cdot [\Phi^*\mc O_{\mb P^d}(1)])\cdot [X]}{{\rm mult}_xX}
	\geq a\frac{[\mc L_{d-1}]\cdot [X]}{{\rm mult}_xX}\geq a\geq {\rm min}\{a, b\}\,. $$
	Next consider the case when $\Phi$ is non-constant on $X$.
	Then there is a section
	$H_2\in H^0(\mc Q, \Phi^*\mc O_{\mb P^d}(1))$ passing through $x$ such that 
	$X$ is not contained in $H_2$. Again we have by B\'ezout's theorem
	$$[\Phi^*\mc O_{\mb P^d}(1)]\cdot [X]=[H_2]\cdot [X]\geq {\rm mult}_xX\,,$$
	which gives 
	$$\frac{(a\cdot [\mc L_{d-1}]
		+b\cdot [\Phi^*\mc O_{\mb P^d}(1)])\cdot [X]}{{\rm mult}_xX}
	\geq b\frac{[\Phi^*\mc O_{\mb P^d}(1)]\cdot [X]}{{\rm mult}_xX}\geq b\geq {\rm min}\{a,b\}\,. $$
	Therefore we get
	$$\varepsilon(a\cdot [\mc L_{d-1}]
	+b\cdot [\Phi^*\mc O_{\mb P^d}(1)],x)\geq {\rm min}\{a, b\}\,.$$
	Now Lemma \ref{lemma-L_{1,x}} implies that
	$$\varepsilon(a\cdot [\mc L_{d-1}]
	+b\cdot [\Phi^*\mc O_{\mb P^d}(1)],x) \leq \dfrac{(a\cdot [\mc L_{d-1}]
		+b\cdot [\Phi^*\mc O_{\mb P^d}(1)])\cdot [L_{1,x}]}{{\rm mult}_xL_{1,x}}=a\,.$$
	Similarly Lemma \ref{lemma-L_{2,x}} implies that 
	$$\varepsilon(a\cdot [\mc L_{d-1}]
	+b\cdot [\Phi^*\mc O_{\mb P^d}(1)],x)\leq \dfrac{(a\cdot [\mc L_{d-1}]
		+b\cdot [\Phi^*\mc O_{\mb P^d}(1)])\cdot [L_{2,x}]}{{\rm mult}_xL_{2,x}}=b\,.$$
	Therefore, we get that for $x\in \mc Z$ we have 
	$$ \varepsilon(a\cdot [\mc L_{d-1}]
	+b\cdot [\Phi^*\mc O_{\mb P^d}(1)],x)={\rm min}\{a,b\}\,.$$
	This completes the proof of the theorem. 
\end{proof}

From the above theorem, we immediately obtain the following results.

\begin{corollary}\label{trivial-cor}
With the notation as in Theorem \ref{trivial-main}, we have 
\begin{enumerate}
\item $\varepsilon(a\cdot [\mc L_{d-1}]
	+b\cdot [\Phi^*\mc O_{\mb P^d}(1)],1) = a$. 
\item $\varepsilon(a\cdot [\mc L_{d-1}]
	+b\cdot [\Phi^*\mc O_{\mb P^d}(1)]) =  {\rm min}\{a,b\}$.
\end{enumerate}
\end{corollary}
\begin{proof} 
This is immediate from the definitions of $\varepsilon(L,1)$ and $\varepsilon(L)$ for a line bundle $L$. 
\end{proof}

\section{Seshadri constants on $\mc Q(E,d)$ when $d>1$}\label{sect2}
We now consider the case of an arbitrary vector bundle $E$ on $C = \mb P^1$. 
We treat the case $d > 1$ in this section. 
In the next section, we deal with the case $d=1$. 

If $E$ is a vector bundle on $C = \mb P^1$, then $E$ is 
a direct sum of line bundles. By tensoring with a suitable line bundle, we may assume that 
$E=\mc O_{C}\oplus \bigoplus\limits^{r-1}_{i=1} \mc O_C(a_i)$ with $0\leq a_i\leq a_j$ for $i<j$.
Let $\mc Q':=\mc Q(E,d)$ be the Quot scheme of torsion quotients 
of $E$ of degree $d$.
First we recall the description of the Nef cone of $\mc Q'$. 
Let $p_C:C\times \mc Q'\to C$ and 
$p_{\mc Q}:C\times \mc Q'\to \mc Q'$ be the projections.  
Let $ p_C^*E\to \mc B'\to 0$ be the 
universal quotient over $C\times \mc Q'$. 
Let $M$ be any line bundle on $C$. 
By \cite[Lemma 3.1]{GS20} 
the sheaf ${p_{\mc Q'}}_*(p_C^*M\otimes \mc B')$ is a vector bundle over
$\mc Q'$ of rank $d$. Define
\begin{equation}
	\mc L_{M,\mc Q'}:={\rm det}({p_{\mc Q'}}_*(p_C^*M\otimes \mc B'))\,.
\end{equation}
The bundle $\mc L_{\mc O,\mc Q'}$ will also
be denoted 
\begin{equation}
	\mc O_{\mc Q'}(1):=\mc L_{\mc O,\mc Q'}={\rm det}({p_{\mc Q'}}_*(\mc B'))\,.
\end{equation}
Let 
\begin{equation}
	L:=\mc O_C(d-1)\,.
\end{equation}
Let 
\begin{equation}
	\Phi':\mc Q'\to S^d\mb P^1\cong \mb P^d
\end{equation}
be the Hilbert-Chow map (see \cite[Section 2]{GS19}).
Then it is proved in \cite[Theorem 6.2]{GS20} that
\begin{align}\label{description-nef-general-bundle}	
	{\rm Nef}(\mc Q') = &  \mb R_{\geq 0}[\mc L_{L,\mc Q'}]+\mb R_{\geq 0}[\Phi'^*\mc O_{\mb P^d}(1)]  \,.                       
\end{align}

\begin{lemma}\label{lemma-L_{1,x}-Q'}
	Given any point $x\in \mc Q'$ there exists a curve 
	$L_{1,x}\cong \mb P^1  \hookrightarrow \mc Q'$ passing through $x$
	such that 
	$[\mc L_{L,\mc Q'}]\cdot [L_{1,x}]
	= 1$ 
	and 
	$[\Phi'^*\mc O_{\mb P^d}(1)]\cdot [L_{1,x}]=0$.
\end{lemma}

\begin{proof}
	Replace $\mc Q$ by $\mc Q'$ and proceed exactly as in the proof of 
	Lemma \ref{lemma-L_{1,x}} till equation
	\eqref{e-2}. Thus, we get that
	${\rm det}(p_{\mc Q'*}(\mc B'\otimes  p_C^*L))=\mc L_{L,\mc Q'}$
	is a globally generated line bundle on $\mc Q'$, and
	we have a map $f:Z\to \mc Q'$ such that $\Phi'\circ f$ is constant.

	In \cite[Theorem 6.2]{GS20} it is proved that 
	$$\mc L_{L,\mc Q'}\cong \mc O_{\mc Q'}(1)\otimes  \Phi'^*\mc O_{\mb P^d}(d-1)\,.$$ 
	Since $\Phi'\circ f$ is constant, it follows that 
	$f^*\mc L_{L,\mc Q'}\cong f^*\mc O_{\mc Q'}(1)\cong \mc O_Z(1)$.
	Consider the composite map $Z\to \mc Q'\to \mb P^N$, where 
	the second map is given by $\mc L_{L,\mc Q'}$. Since the pullback
	of $\mc O_{\mb P^N}(1)$ along this map is $\mc O_Z(1)$ it follows
	that the composite, and so also $f$, is a closed immersion.
		
	Let $L_{1,x}\subset Z\subset \mc Q'$ denote any line passing through $x$.
	Then it is clear that $L_{1,x}$ satisfies the assertions of the lemma. 
\end{proof}  

Let $E'$ denote the largest proper sub-bundle of $E$ in the Harder-Narasimhan filtration 
of $E$. Then $E/E'$ is the trivial bundle. 

\begin{lemma}\label{lemma-d>1-E}Assume $d>1$. Let $E \to B$ be a 
	torsion quotient of degree $d$ such that the image of $H^0(C,E)$ 
	in $H^0(C,B)$ is one-dimensional.  
	Then this quotient factors through $E\to E/E'$. \end{lemma}
\begin{proof}
	Let $V$ denote the vector space $H^0(C,E)$. Let $\mc O(a)$ be a sub-bundle 
	of $E'$. Then $a>0$. To show that $\mc O(a)$ is mapped to 0 
	by the given projection $E \to B$, 
	it suffices 
	to show that the global sections are mapped to 0. If not, then there is a 
	commutative diagram
	\[\xymatrix{
		\mc O\ar[r]\ar[d] & H^0(C,\mc O(a))\otimes \mc O\ar[r] & V \otimes \mc O\ar[r]\ar[d] & \mc O\ar[d]\\
		\mc O(a)\ar[rr] && E\ar[r] & B 
	}
	\]
	The square on the right exists because of the assumption that 
	the image $H^0(C,E)\to H^0(C,B)$ is one-dimensional. By assumption
	the top horizontal composite map is nonzero. Since the right 
	vertical arrow is a surjection, 
	it follows that the composite map $\mc O(a)\to B$ is surjective. 
	We may choose homogeneous coordinates $X$ and $Y$ on $C = \mathbb{P}^1,$ 
	so that the support 
	of $B$ is contained in the open set $Y\neq 0$. Let $T$ denote the 
	affine coordinate $X/Y$. Thus, the map 
	$\mc O(a)\to B$ can be thought of as a surjective map $k[T]\to k[T]/(p(T))$
	of $k[T]$ modules. 
	Viewed like this, a basis for the global sections of $\mc O(a)$ is given by the functions $T^i$ where $0\leq i\leq a$. 
	The degree of $p(T)$ is equal to the length of the module $k[T]/(p(T))\cong B$, 
	which is exactly $d$. Since $d>1$, we get a contradiction to 
	the assumption that the image of $H^0(C,\mc O(a))$ is one-dimensional. 
\end{proof}

There is a natural inclusion $j':\mc Q(E/E',d)\hookrightarrow \mc Q'$ which 
commutes with the Hilbert-Chow maps. Since the bundle $E/E'$ is trivial,
$\mc Q(E/E',d)=\mc Q(n'',d)=:\mc Q''$. Let $n={\rm dim}(H^0(C,E))$ 
and let $\mc Q:=\mc Q(n,d)$. The surjection $H^0(C,E)\otimes \mc O_C\to E$
defines an inclusion $j:\mc Q'\hookrightarrow \mc Q$.
The following diagram is commutative. 

\[\xymatrix{
	\mc Q''\ar@{^(->}[r]^{j'}\ar[dr]_{\Phi''} & 
		\mc Q'\ar@{^(->}[r]^{j}\ar[d]^{\Phi'} &\mc Q\ar[dl]^{\Phi}\\
	&\mb P^d
}
\]
The $\Phi$'s denote the Hilbert-Chow maps.
In particular there are subsets 
$\mc U''\subset \mc Q''$ and $\mc Z''\subset \mc Q''$ as in 
Definition \ref{def-Z}. 
\begin{corollary}
	Assume $d>1$. Then $\mc Q'\setminus j^{-1}(\mc U)$
	is precisely $j'(\mc Z'')$. 
\end{corollary}
\begin{proof}
	Follows from Lemma \ref{lemma-d>1-E}.
\end{proof}

\begin{lemma}\label{lemma-L_{2,x}-Q'}
	Assume $d>1$.
	Given any point $x\in \mc Q'\setminus j^{-1}(\mc U)$ there exists a curve 
	$L_{2,x} \cong \mb P^1 \hookrightarrow \mc Q'$ passing through $x$ 
	such that 
	$[\mc L_{L,\mc Q'}]\cdot [L_{2,x}]=0$ and
	$[\Phi''^*\mc O_{\mb P^d}(1)]\cdot [L_{2,x}]=1$.
\end{lemma}

\begin{proof}
	We have maps $\mc Q''\xrightarrow{j'} \mc Q'\xrightarrow{j}\mc Q$.
	As observed above, $x$ is in the image of $\mc Z''$.
	It can be shown that $j^*\mc L_{d-1}=\mc L_{L,\mc Q'}$, for this see the 
	proof of \cite[Theorem 6.2]{GS20}.
	Thus, the Lemma is clear using Lemma \ref{lemma-L_{2,x}} once we observe that 
	$$j'^*\mc L_{L,\mc Q'}\cong j'^*j^*\mc L_{d-1}\cong \mc L_{d-1,\mc Q''}\,.$$
	This is easily seen using \cite[Lemma 3.14]{GS20}.
\end{proof}

\begin{theorem}\label{general-main}
	Assume  $d>1$.
	Fix integers $a,b>0$. If $x\in \mc Q'$ such that $j(x)\in \mc U$, then we have 
	$$\varepsilon(a\cdot [\mc L_{L,\mc Q'}]
	+b\cdot [\Phi'^*\mc O_{\mb P^d}(1)],x)=a.$$ 
	If $x\in \mc Q'\setminus \mc U$ we have
	$$\varepsilon(a\cdot [\mc L_{L,\mc Q'}]
	+b\cdot [\Phi'^*\mc O_{\mb P^d}(1)],x)= {\rm min}\{a,b\}.$$
\end{theorem}

\begin{proof}
	Let $x\in \mc Q'$ such that $j(x)\in \mc U$. 
	Let $X$ be any irreducible and reduced curve in $\mc Q'$ passing through $x$.
	Consider the composite map 
	$$X\hookrightarrow \mc Q'\xrightarrow{j} \mc Q \xrightarrow{\Psi} \mb P^N\,.$$
	Then by Lemma \ref{lemma-an injective map from an open set} the map
	$\Psi(j(X)\cap \mc U)\neq {\rm pt}$. Hence there exists a section
	$H^0(\mc Q',\mc L_{L,\mc Q'})$ whose zero locus $H$ passes 
	through $x$ such that $X$ is not
	contained in $H$. By B\'ezout's theorem, we get     
	$$[\mc L_{L,\mc Q'}]\cdot [X]
	=[H]\cdot [X]\geq ({\rm mult}_xH)({\rm mult}_xX)\geq {\rm mult}_xX\,.$$
	Hence    
	$$(a\cdot [\mc L_{L,\mc Q'}]+b\cdot [\Phi'^*\mc O_{\mb P^d}(1)])\cdot [X]
	\geq a\cdot [\mc L_{L,\mc Q'}] \cdot [X]\geq a\cdot {\rm mult}_xX\,.$$
	Therefore, $\varepsilon(a\cdot [\mc L_{L,\mc Q'}]
	+b\cdot [\Phi'^*\mc O_{\mb P^d}(1)],x)\geq a$.
	Now  by Lemma \ref{lemma-L_{1,x}-Q'} we have 
	$$\varepsilon(a\cdot [\mc L_{L,\mc Q'}]
	+b\cdot [\Phi'^*\mc O_{\mb P^d}(1)],x)
	\leq \dfrac{(a\cdot [\mc L_{L,\mc Q'}]
		+b\cdot [\Phi'^*\mc O_{\mb P^d}(1)])\cdot [L_{1,x}]}{{\rm mult}_xL_{1,x}}
	=a\,.$$
	Hence, we get for $x\in \mc U$
	$$\varepsilon(a\cdot [\mc L_{L,\mc Q'}]
	+b\cdot [\Phi'^*\mc O_{\mb P^d}(1)],x)=a\,.$$    
	
	Now let $j(x)\in \mc Z$.
	We first prove the inequality 
	$$\varepsilon(a\cdot [\mc L_{L,\mc Q'}]
	+b\cdot [\Phi'^*\mc O_{\mb P^d}(1)],x)\geq {\rm min}\{a,b\}\,.$$  
	Let $X$ be any irreducible and reduced curve in $\mc Q'$ 
	passing through $x$.  
	We have maps $\Psi\circ j:\mc Q' \to \mc Q\to \mb PH^0(\mc Q,\mc L_{d-1})$ 
	and $\Phi':\mc Q'\to \mb P^d$.
	The class $[\mc L_{L,\mc Q'}]+[\Phi'^*\mc O_{\mb P^d}(1)]$ is ample,
	thus, it cannot happen that  
	$\Psi$ and $\Phi'$ both are constant on $X$. 
	First consider the case when $\Psi$ is non-constant on $X$.
	Hence, there 
	exists a section 
	$H_1\in H^0(\mc Q',\mc L_{L,\mc Q'} )$ passing 
	through $x$ such that $X$ is not contained in $H_1$. 
	By B\'ezout's theorem, we will have
	$$[\mc L_{L,\mc Q'}]\cdot [X]
	=[H_1]\cdot [X]\geq {\rm mult}_xX\,,$$
	which gives 
	$$\frac{(a\cdot [\mc L_{L,\mc Q'}]
		+b\cdot [\Phi'^*\mc O_{\mb P^d}(1)])\cdot [X]}{{\rm mult}_xX}
	\geq a\geq {\rm min}\{a,b\}\,. $$
	Next consider the case when $\Phi'$ is non-constant on $X$.
	Then there is a section
	$H_2\in H^0(\mc Q', \Phi'^*\mc O_{\mb P^d}(1))$ passing through $x$ such that 
	$X$ is not contained in $H_2$. Again we have by B\'ezout's theorem
	$$[\Phi'^*\mc O_{\mb P^d}(1)]\cdot [X]=[H_2]\cdot [X]\geq {\rm mult}_xX\,,$$
	which gives 
	$$\frac{(a\cdot [\mc L_{L,\mc Q'}]
		+b\cdot [\Phi'^*\mc O_{\mb P^d}(1)])\cdot [X]}{{\rm mult}_xX}
	\geq b\geq {\rm min}\{a,b\}\,. $$
	Therefore we get
	$$\varepsilon(a\cdot [\mc L_{L,\mc Q'}]
	+b\cdot [\Phi'^*\mc O_{\mb P^d}(1)],x)\geq {\rm min}\{a,b\}\,.$$
	Now Lemma \ref{lemma-L_{1,x}-Q'} implies that
	$$\varepsilon(a\cdot [\mc L_{d-1}]
	+b\cdot [\Phi^*\mc O_{\mb P^d}(1)],x) \leq \dfrac{(a\cdot [\mc L_{d-1}]
		+b\cdot [\Phi^*\mc O_{\mb P^d}(1)])\cdot [L_{1,x}]}{{\rm mult}_xL_{1,x}}=a\,.$$
	Similarly Lemma \ref{lemma-L_{2,x}-Q'} implies that 
	$$\varepsilon(a\cdot [\mc L_{d-1}]
	+b\cdot [\Phi^*\mc O_{\mb P^d}(1)],x)\leq \dfrac{(a\cdot [\mc L_{d-1}]
		+b\cdot [\Phi^*\mc O_{\mb P^d}(1)])\cdot [L_{2,x}]}{{\rm mult}_xL_{2,x}}=b\,.$$
	Therefore, we get that for $x\in \mc Z$ we have 
	$$ \varepsilon(a\cdot [\mc L_{d-1}]
	+b\cdot [\Phi^*\mc O_{\mb P^d}(1)],x)={\rm min}\{a,b\}\,.$$
	This completes the proof of the theorem. 
\end{proof}

As before, we immediately obtain the following corollary. 

\begin{corollary}\label{general-cor}
With the notation as in Theorem \ref{general-main}, we have 
\begin{enumerate}
\item $\varepsilon(a\cdot [\mc L_{L,\mc Q'}]
	+b\cdot [\Phi'^*\mc O_{\mb P^d}(1)],1) = a$. 
\item $\varepsilon(a\cdot [\mc L_{L,\mc Q'}]
	+b\cdot [\Phi'^*\mc O_{\mb P^d}(1)]) =  {\rm min}\{a,b\}$.
\end{enumerate}
\end{corollary}

\section{Seshadri constants on $\mc Q(E,1)$}\label{sect3}
Let $Y$ be a smooth projective complex curve and $E$ a vector
bundle on $Y$. Let $X = \mathbb{P}(E)$ be the projective
bundle associated to $E$ over $Y$. The Quot scheme $\mc Q(E,1)$
is isomorphic to $X$. Denote by $\xi:=\mc O_{\mb P(E)}(1)$ 
and by $\mf f$ the divisor which is a fibre of $\pi$.

Let $Q$ denote a vector bundle quotient of $E$ with the smallest 
slope {and largest rank}. Note that
if $$0=E_0 \subset E_1 \subset \ldots \subset E_{d-1} \subset E_d = E$$ is the
Harder-Narasimhan filtration of $E$, then $Q = E/E_{d-1}$. Let $e : =
\mu(Q)$ denote the slope of $Q$. 

\begin{theorem}[Miyaoka] The nef cone of $X$ is spanned by $\xi-e\mf f$
	and $\mf f$. 
\end{theorem}

\noindent 
See \cite{M85} and \cite[Lemma 2.1]{F11}. 
The dual basis of the closed cone of curves $\overline{NE}(X)$ is
given by extremal rays spanned by two $1$-cycles $C_1, C_2$, where
$C_2$ denotes the class of a line in a fibre of $\pi$ (which is
isomorphic to the projective space $\mathbb{P}^{{\rm rk}(E)-1}$). The other
1-cycle $C_1$ is not effective in general. {However, $C_1$ is the pushforward of 
a pseudoeffective 1-cycle via the natural embedding 
$\mb P(Q) \to \mb P(E)$; see \cite[Lemma 2.3]{F11}.  }

We now assume $Y = \mathbb{P}^1$. 
If $E$ is a trivial vector bundle then $\mb P(E) = \mb P^1 \times \mb P^{\text{rk}(E)-1}$. In this case, it 
is easy to compute the Seshadri constants of ample line bundles on $\mb P(E)$; for example, 
see \cite[Proposition 3.4(e)]{MR15}. 

Now let $E$ be a bundle 
such that $E\not\cong \mc O^{{\rm rk}(E)}$. Since 
$\mathbb{P}(E) \cong \mathbb{P}(E\otimes L)$ for any line bundle
$L$ on $Y$, we may assume $E = \mc{O}^s \oplus 
\mc O(a_1)\ldots \oplus \mc{O}(a_r)$ with $0<a_1 \le \ldots \le a_r$
(after tensoring $E$ with a suitable line
bundle). 
Then the quotient of $E$ with the smallest slope is 
$Q = \mc O^s$. In other words, 
$$E_{d-1} =  \underset{i \geq 1}\oplus \mc{O}(a_i).$$
So we have $e = \mu(Q) = 0$. 
In this case, the extremal ray $C_1$, in the above notation, is
spanned by the image of the section $$\mathbb{P}^1 \to X$$
corresponding to a rank 1 quotient 
$$E \to \mc{O}\,.$$
Indeed, first observe that $\mf f \cdot C_2 = 0$ and $\mf f\cdot C_1 = 1$. 
Next note that $\xi \cdot C_1 = \text{deg}(\mc{O})=0$ and  $\xi \cdot C_2 = 1$. 
So $C_2,C_1$ is dual to $\xi, \mf f$.

We first prove a general result about Seshadri constants on $X$.

\begin{proposition}\label{general} Let $L$ be an ample line bundle on $X$ which is numerically
	equivalent to the bundle $a\xi+b\mf f$, where $a,b$ are positive
	integers. Then {$a \ge$} $\varepsilon(X,L,x) \ge \textrm{min}\{a,b\}$ for all
	$x \in X$.
\end{proposition}
\begin{proof} 
{To prove the first inequality, let $x \in X$. Then $x \in \pi^{-1}(\pi(x)) = \mb P^{r+s-1}$. Choose a line 
$l$ in $\mb P^{r+s-1}$ containing $x$. The ratio given by this line is $\frac{L \cdot l}{1} = a$. Hence 
$a \ge \varepsilon(X,L,x)$.} 

{Now we prove the second inequality.}
Let $C \subset X$ be an irreducible and reduced curve
	such that $m: =\text{mult}_x(C) > 0$. Write $C = pC_1+qC_2$ 
	for non-negative integers $p,q$.\\
	
	\noindent 
	\textbf{Case 1}: Suppose $\pi(C)$ is a point $y \in
	\mathbb{P}^1$. Then $p=0$ and $C$ is a curve in $\pi^{-1}(y)$ of
	degree $q$. Thus, $C\subset \mb P^{r+s-1}$ and if $H$ denotes 
	a general hyperplane through $x$ then 
	$$q=C\cdot H\geq ({\rm mult}_xC)({\rm mult}_xH)\geq m\,.$$
	So 
	$$\frac{L\cdot C}{m} = \frac{aq}{m} \ge a\geq {\rm min}\{a,b\}\,.$$
	
	\noindent 
	\textbf{Case 2}: Suppose  $\pi(C) = \mathbb{P}^1$. Let $W=
	\pi^{-1}(\pi(x))$. Since $C \not\subset W$, B\'ezout's theorem gives
	$$p = \mf f \cdot C = W\cdot C \ge ({\rm mult}_xC)({\rm mult}_xW)\geq m.$$
	Hence $$\frac{L\cdot C}{m} = \frac{aq+bp}{m} \ge b\geq {\rm min}\{a,b\}\,.$$
	So $\frac{L\cdot C}{m} \ge  \textrm{min}\{a,b\}$ for all curves $C$
	passing through $x$. This gives the desired bound 
	$\varepsilon(X,L,x) \ge \textrm{min}\{a,b\}$ for all
	$x \in X$.
\end{proof}

Now we obtain more precise values of the Seshadri constants on $X$. 
Recall that $E/E_{d-1}=\mc O^s$. We get a morphism $i:\mb P(E/E_{d-1})\hookrightarrow X$
using the quotient 
$$E\to E/E_{d-1}\to \mc O_{\mb P(E/E_{d-1})}(1)\,.$$ 
Let $Z$ denote the image of $i$. We have 
$Z \cong \mathbb{P}^1 \times \mathbb{P}^{s-1}$ and 
$i^*\xi=p_2^*\mc O_{\mb P^{s-1}}(1)$ and $i^*\mf f=p_1^*\mc O_{\mb P^1}(1)$
where $p_i$ are the projections from $\mb P^1\times \mb P^{s-1}$.

\begin{theorem}\label{main-projective-bundle}
	Let $L$ be an ample line bundle on $X$ which is numerically
	equivalent to the bundle $a\xi+b\mf f$, where $a,b$ are positive
	integers. Let $x \in X$. We have
	\begin{equation*}
		\varepsilon(L,x) \,=\, 
		\begin{cases}
			a, & \text{~if~}  x \notin Z, \\
			{\rm min}\{a,b\}, & \text{~if~} x \in Z.
		\end{cases}
	\end{equation*}
\end{theorem}
\begin{proof}
	Suppose first that $x \in Z$. The restriction of $L$ to $Z$ is the
	bundle $L_Z: = p_1^*\mc{O}_{\mb P^1}(a) \otimes
	p_2^*\mc{O}_{\mb P^{s-1}}(b)$. 
	It is easy to see that $\varepsilon(Z,L_Z,x) = \text{min}\{a,b\}$ (for
	example, see \cite[Proposition 3.4(e)]{MR15}). Obviously 
	$\varepsilon(X,L,x) \leq \varepsilon(Z,L_Z,x)$. By 
	Proposition \ref{general}, it follows that 
	$\varepsilon(X,L,x) = \text{min}\{a,b\}$. 
	
	Now assume $x \notin Z$. 
	Let $C \subset X$ be an irreducible and reduced
	curve such that $m: = \text{mult}_x(C) > 0$. So we have $C \not\subset
	Z$. 
	Let $\phi: C\to \mb P^1$ be the composition of the inclusion $C \subset
	X$ with the natural map $X \to \mb P^1$. 
	Let $\phi^*E\to M$ be the line bundle quotient which defines the map
	$C\to \mb P(E)$. 
	Then since $C$ is not contained in $Z$, the quotient map $\phi^*E\to M$
	does not factor through $\phi^*E/\phi^*E_{d-1}$. In other words, the
	composition map 
	$\phi^*E_{d-1} \to \phi^*E\to M$ is not zero. Recall 
	$E_{d-1} = \underset{i \geq 1}\oplus \mc{O}(a_i)$ and all $a_i>0$. Thus, 
	the composition
	$\phi^*\mc{O}(a_i)\subset \phi^*E_{d-1}\to M$ is non-zero for some $i$. 
	Write $C = pC_1+qC_2$ for non-negative integers $p,q$. Note that 
	$p=C\cdot \mf f={\rm deg}(\phi)$. Thus, we get
	$$q=\xi \cdot C = {\rm deg}(\xi|_C)={\rm deg}(M)\geq a_i{\rm deg}(\phi)=a_ip\geq p\,.$$ 
	Suppose that $\pi(C)$ is a point in $\mathbb{P}^1$.
	Then by Case 1 of Proposition \ref{general}, we have 
	$$\frac{L\cdot C}{m} = \frac{aq}{m} \ge a.$$
	Next suppose  that $\pi(C) = \mathbb{P}^1$. Then arguing as in Case 2 of
	Proposition \ref{general}, we conclude $p \ge m$. So $q \ge p \ge m$
	and $$\frac{L\cdot C}{m} = \frac{aq+bp}{m} \ge \frac{(a+b)p}{m} \ge
	a+b \ge a.$$
	So $\frac{L\cdot C}{m} \ge  a$ for all curves $C$
	passing through $x$. This gives $\varepsilon(X,L,x) \ge a$. 
	
	On the other hand, let $W= \pi^{-1}(\pi(x))$. Let $D
	\subset W \cong \mb P^{r-1}$ be a line containing  $x$. Then $D$ is
	smooth and $L\cdot D = a$. So $\varepsilon(X,L,x) \le a$. 
	
	We conclude that $\varepsilon(X,L,x)  = a$ for all $x
	\notin Z$, as required. 
\end{proof}

\begin{corollary}\label{cor-projective-bundle}
With the notation as in Theorem \ref{main-projective-bundle}, we have 
\begin{enumerate}
\item $\varepsilon(L,1) = a$. 
\item $\varepsilon(L) =  {\rm min}\{a,b\}\,$.
\end{enumerate}
\end{corollary}

\begin{remark}
	If $s=1$ in the above notation, Theorem \ref{main-projective-bundle} 
	follows directly from \cite[Theorem 3.3]{BHNN21}.  
	If $s=1$ then the rank of $E_d/E_{d-1}$ is 1. We take $m=d, r=1$ in the
	notation of \cite[Section 2.1]{BHNN21}. So the Grassmann bundle is nothing but the
	projective bundle $\mathbb{P}(E)$. Then the sub variety $Z =
	\mathbb{P}(E/E_{d-1}) \subset X$ defined above coincides with the
	section $\Gamma_s$ defined in \cite[Section 3]{BHNN21}. 
	The other hypotheses of
	\cite[Theorem 3.3]{BHNN21} can be verified easily to obtain Theorem \ref{main-projective-bundle}.  
\end{remark}

\section{Seshadri constants on $Gr(E,n)$}\label{sect4}
In this section, we compute Seshadri constants for ample line bundles on the Grassmann bundle over 
$C = \mb P^1$. For a vector bundle $E$ on $\mb P^1$ and a positive integer $n$, recall that 
$Gr(E,n)$ denotes the Grassmann bundle that parametrizes
the $n$-dimensional quotients of the 
fibers of $E$.

Given a vector bundle $E$ on $\mb P^1$, we tensor $E$ with a suitable line bundle and assume that 
$E = \mc{O}^{r_0} \oplus  \mc O(a_1)^{r_1}\oplus \ldots \oplus \mc{O}(a_m)^{r_m}$, with $0<a_1<\ldots< a_m$
and $r_i>0$ for all $0 \le i \le m$. Define $a_0=0$ and $r_{-1}=0$. 
The least possible
degree among all possible quotients of $E$ of rank $n$ is 
described as follows. Let $t$ be the smallest integer such that 
$$r_0+\ldots+r_t\geq n\,.$$
Define $l:=r_0+\ldots+r_{t-1}$; then $l<n$. Then the quotient of $E$ of 
least possible degree is isomorphic to 
$$\left(\bigoplus_{i=0}^{t-1}\mc O(a_i)^{r_i}\right)\oplus \mc O(a_t)^{n-l}\,.$$
Clearly, this has degree $d_0:=(n-l)a_t+\sum_{i=1}^{t-1}r_ia_i$.
Consider the Pl\"ucker 
embedding $i:Gr(E,n)\subset \mb P(\wedge^nE)$. 
Let $\pi:\mb P(\wedge^nE)\to \mb P^1$ denote the projection
and let $\pi':=\pi\circ i$. 
Then the above quotient defines a section $s$
of $\pi'$ and $i\circ s$ defines a section of $\pi$. 
One checks easily that the nef cone of $\mb P(\wedge^nE)$
has as boundaries $\mf f=\pi^*\mc O_{\mb P^1}(1)$ and 
$\xi=\mc O_{\mb P(\wedge^nE)}(1)\otimes \pi^*\mc O_{\mb P^1}(-d_0)$.
The latter line bundle is {strictly} nef as $(\wedge^nE)\otimes \mc O_{\mb P^1}(-d_0)$
is globally generated and has the trivial bundle as a quotient.
Using the section $s$ one easily checks that the nef cone
of $\mb P(\wedge^nE)$ maps onto the nef cone of $Gr(E,n)$.
See \cite{Biswas-Param} for a description of the Nef cones of flag varieties 
over any curve.

Let $Z:=Gr(\mc O(a_t)^{r_t},n-l)$. Let $\pi'':Z\to \mb P^1$
denote the canonical map. Then we have the tautological quotient 
$$q_Z:\pi''^*\mc O(a_t)^{r_t}\to F_Z$$
on $Z$.
Denote by $E'$ and 
$E''$ the following summands of $E$,
$$E'=\bigoplus_{i=0}^{t-1}\mc O(a_i)^{r_i} \qquad\text{and}\qquad
E''=\bigoplus_{i=t+1}^{l}\mc O(a_i)^{r_i}\,.$$
Thus, $E=E'\oplus \mc O(a_t)^{r_t}\oplus E''$. We get 
the following quotient on $Z$
\begin{align*}
	\pi''^*E=\pi''^*E'\oplus \pi''^*\mc O(a_t)^{r_t}\oplus \pi''^*E''\xrightarrow{Id\oplus q_Z\oplus 0}
		\pi''^*E'\oplus F_Z\oplus 0\,,
\end{align*}
which defines a map $Z\to Gr(E,n)$.

\begin{lemma}
	Let $C\subset Gr(E,n)$ be a curve such that 
	$\xi\cdot C=0$. Then this inclusion factors as
	$C\to Z\to Gr(E,n)$.
\end{lemma}
\begin{proof}
Let 
$\phi:C\to \mb P^1$ denote the projection to $\mb P^1$.
Let $q:\phi^*E\to Q$ denote the quotient which defines 
the map $C\to Gr(E,n)$.
Note that if $a:Q\to Q$ is an isomorphism, then the quotients
$q$ and $a\circ q$ define the same maps from $C\to Gr(E,n)$.
We will find an $a$ for which it is clear that the map 
defined by $a\circ q$ factors through $Z$. 

Since $\xi\cdot C=0$ it is clear that 
${\rm deg}({\rm det}(Q))=d_0{\rm deg}(\phi)$.
The bundle $\phi^*(\wedge^nE)$ is a direct sum of line bundles,  
each of which has degree at least $d_0{\rm deg}(\phi)$.
Since a line bundle cannot map to
a line bundle of strictly lower degree, it follows
that there is a line bundle summand of $\phi^*(\wedge^nE)$ 
of degree $d_0{\rm deg}(\phi)$
which maps isomorphically onto ${\rm det}(Q)$.
From this we easily conclude that $Q\cong \phi^* E'\oplus \phi^*\mc O(a_t)^{n-l}$
and that there is a summand of 
$\phi^*E$ 
such that the composite
$$\phi^* E'\oplus \phi^*\mc O(a_t)^{n-l}\subset \phi^* E \xrightarrow{q} \phi^* E'\oplus \phi^*\mc O(a_t)^{n-l}$$
is an isomorphism. 
Denote the above isomorphism by $a$. By looking at degree
it is clear that the summand $\phi^*E''\subset \phi^*E$
maps to 0 under $q$. Consider the 
map $a^{-1}\circ q$. It follows that there is a quotient 
$$q':\phi^*\mc O(a_t)^{r_t}\to \phi^*\mc O(a_t)^{n-l}$$
such that $a^{-1}\circ q$ is of the type

\begin{align*}
	\phi^*E=\phi^*E'\oplus \phi^*\mc O(a_t)^{r_t}\oplus \phi^*E''\xrightarrow{Id\oplus q'\oplus 0}
		\phi^* E'\oplus \phi^*\mc O(a_t)^{n-l}\oplus 0\,.
\end{align*}

The quotient $q'$ defines a map $C \to Z$ and it is 
clear that the map $C\to Gr(E,n)$ factors as $C \to Z\to Gr(E,n)$.
\end{proof}

\begin{theorem}\label{main-grassmann}
	Let $L$ be an ample line bundle on $Gr(E,n)$ which is numerically
	equivalent to the bundle $a\xi+b\mf f$, where $a,b$ are positive
	integers. Let $x \in Gr(E,n)$. We have
	\begin{equation*}
		\varepsilon(L,x) \,=\, 
		\begin{cases}
			a, & \text{~if~}  x \notin Z, \\
			{\rm min}\{a,b\}, & \text{~if~} x \in Z.
		\end{cases}
	\end{equation*}
\end{theorem}
\begin{proof}
If $x\notin Z$ then for each curve
$C$ through $x$ we have $\xi\cdot C>0$. Since $\xi$
is a globally generated line bundle we have a map
$\psi:Gr(E,n)\to \mb PH^0(Gr(E,n),\xi)$. Let $C$ be a curve
through $x$. Then  since $\xi\cdot C>0$ it follows
that $\psi(C)$ is a curve passing through $\psi(x)$.
There is a hyperplane through $\psi(x)$ such that the 
intersection with $\psi(C)$ is proper. Thus, there 
is a global section of $\xi$ whose vanishing locus
is a divisor $H$ through $x$ and its intersection with 
$C$ is proper. This shows, using B\'ezout's theorem, that 
$$\xi\cdot C \geq {\rm mult}_x(H){\rm mult}_x(C)\geq {\rm mult}_x(C)\,.$$
Thus,
$$\frac{(a\xi+b\mf f)\cdot C}{{\rm mult}_x(C)}\geq a$$
for all curves $C$ passing through $x$. By taking 
$C$ to be a line in the fiber of $\pi'$ we see that 
the ratio $a$ is attained. This shows that 
$$\varepsilon(Gr(E,n),a\xi+b\mf f,x)= a\,.$$

Next consider the case when $x\in Z$. Note 
that $Z\cong \mb P^1 \times {\rm Gr}(r_t,n-l)$. 
It is easily checked that the restriction of $\xi$
to $Z$ is $\mc O_{{\rm Gr}(r_t,n-l)}(1)$.
Thus, through every point of $Z$ there is a section
$s$ of $\pi'$ such that $\xi\cdot s(\mb P^1)=0$ 
and $\mf f\cdot s(\mb P^1)=1$. Thus, the ratio
$$\frac{(a\xi+b\mf f)\cdot s(\mb P^1)}{{\rm mult}_x(s(\mb P^1))}=b$$
is attained. By Theorem \ref{main-projective-bundle}
$$\varepsilon(Gr(E,n),a\xi+b\mf f,x)\geq \varepsilon(\mb P(\wedge^nE),a\xi+b\mf f,x) \geq {\rm min}\{a,b\}\,.$$
This shows that $\varepsilon(Gr(E,n),a\xi+b\mf f,x)= b$.
This completes the proof of the theorem.
\end{proof}

\begin{corollary}\label{cor-grassmann}
With the notation as in Theorem \ref{main-grassmann}, we have 
\begin{enumerate}
\item $\varepsilon(L,1) = a$. 
\item $\varepsilon(L) =  {\rm min}\{a,b\}\,$.
\end{enumerate}
\end{corollary}

\begin{remark}
	Seshadri constants on Grassmann bundles $Gr(E,n)$ over \textit{arbitrary} smooth curves are studied in \cite{BHNN21}. However, 
	they only consider Grassmann bundles corresponding to rank $n$ quotients for certain specific values of $n$, which are 
	determined by the Harder-Narasimhan filtration of $E$. In Theorem \ref{main-grassmann}, we do not impose any conditions 
	on $n$ and determine Seshadri constants for any line bundle on a Grassmann bundle over $\mb P^1$. 
\end{remark}

\bibliographystyle{halpha}
\bibliography{./sc}

\end{document}